\documentclass{birkjour} 

 \usepackage{amsopn}
 \usepackage{amsmath,amsthm,amssymb}

 \newcommand{\nc}{\newcommand}

\nc{\bb}{\mathfrak{b} }
 \nc{\cc}{\mathfrak{c} }  \nc{\dd}{\mathfrak{d} } 
    \nc{\ggo}{\mathfrak{g} }
 \nc{\hh}{\mathfrak{h} }  \nc{\ii}{\mathfrak{i} }
 \nc{\jj}{\mathfrak{j} }  \nc{\kk}{\mathfrak{k} }
\nc{\mm}{\mathfrak{m} }   \nc{\nn}{\mathfrak{n} }
\nc{\pp}{\mathfrak{p} }   
\nc{\rr}{\mathfrak{r} } \nc{\sg}{\mathfrak{s} }
 \nc{\sso}{\mathfrak{so} }  \nc{\spg}{\mathfrak{sp} }
 \nc{\ssu}{\mathfrak{su} }  \nc{\ssl}{\mathfrak{sl} }
 \nc{\tog}{\mathfrak{t} }  \nc{\uu}{\mathfrak{u} }
 \nc{\vv}{\mathfrak{v} } \nc{\ww}{\mathfrak{w} }
 \nc{\zz}{\mathfrak{z} }  
 \nc{\iso}{\mathfrak{iso}}

\nc{\CC}{{\mathbb C}}
 \nc{\DD}{{\mathbb D}}
\nc{\FF}{{\mathbb F}}
\nc{\GG}{{\mathbb G}}  
\nc{\HH}{{\mathbb H}}
\nc{\II}{{\mathbb I}}
\nc{\JJ}{{\mathbb J}}
\nc{\KK}{{\mathbb K}}
\nc{\NN}{{\mathbb N}}

\nc{\RR}{{\mathbb R}}  
 \nc{\ZZ}{{\mathbb Z}}

\nc{\ggob}{\overline{\mathfrak{g}}} 
 
\nc{\glg}{\mathfrak{gl} }
  
\nc{\pca}{\mathcal{P}} \nc{\nca}{\mathcal{N}}
 
 \nc{\vp}{\varphi} \nc{\ddt}{\frac{{\rm d}}{{\rm d}t}}
 \nc{\la}{\langle} \nc{\ra}{\rangle}
 \nc{\brg}{[\,,\,]_{\ggo}}
 \nc{\brv}{[\,,\,]_{\vv}}
 
 \nc{\SO}{{\sf SO}} \nc{\Spe}{{\sf Sp}} \nc{\Sl}{{\sf Sl}}
 \nc{\SU}{{\sf SU}} \nc{\Or}{{\sf O}} \nc{\U}{{\sf U}}
 \nc{\Gl}{{\sf Gl}} \nc{\Se}{{\sf S}} \nc{\Cl}{{\sf Cl}}
 \nc{\Spin}{{\sf Spin}} \nc{\Pin}{{\sf Pin}}

 \nc{\ad}{\operatorname{ad}} \nc{\Ad}{\operatorname{Ad}}
 \nc{\coad}{\operatorname{coad}} 
 \nc{\rank}{\operatorname{rank}} \nc{\Irr}{\operatorname{Irr}}
 \nc{\End}{\operatorname{End}} \nc{\Aut}{\operatorname{Aut}}
 \nc{\Inn}{\operatorname{Inn}} \nc{\Der}{\operatorname{Der}}
 \nc{\Hh}{\operatorname{H}}
 \nc{\Ker}{\operatorname{Ker}} \nc{\Iso}{\operatorname{Iso}}
 \nc{\Le}{\operatorname{L}} \nc{\tr}{\operatorname{tr}}
 \nc{\dif}{\operatorname{d}} \nc{\sen}{\operatorname{sen}}
 \nc{\modu}{\operatorname{mod}} \nc{\Ric}{\operatorname{Ric}}
 \nc{\Sym}{\operatorname{Sym}} \nc{\sca}{\operatorname{sc}}
 \nc{\scalar}{{\sf s}} \nc{\grad}{\operatorname{grad}}
 \nc{\Rc}{\operatorname{Rc}} 
 \nc{\Lie}{\operatorname{L}} \nc{\ct}{\operatorname{T}}

\newcommand{\deax}{\partial_x}
\newcommand{\deay}{\partial_y}
\newcommand{\deaz}{\partial_z}
\newcommand{\deat}{\partial_t}

\newcommand{\betaidot}{\stackrel{\cdot}{\beta_i}}
\newcommand{\betaidotdot}{\stackrel{\cdot\cdot}{\beta_i}}
\newcommand{\alphajdot}{\stackrel{\cdot}{\alpha_j}}

 \newcommand{\meti}{\left\langle}
 \newcommand{\metd}{\right\rangle}
\newcommand{\lela}{\left \langle}
\newcommand{\rira}{\right \rangle}
\newcommand{\bil}{\lela\,,\,\rira}

\nc{\mr}{{\mathfrak r}}
\nc{\ms}{{\mathfrak s}}
\nc{\mv}{{\mathfrak v}}
\nc{\lra}{\longrightarrow}
\nc{\R}{{\mathbb R}}
\nc{\Z}{{\mathbb Z}}

\newcommand{\mn}{\mathfrak n }
\newcommand{\mz}{\mathfrak z }

\newcommand{\mgg}{\mathfrak g }

 \theoremstyle{plain}
 \newtheorem{thm}{Theorem}[section]
 \newtheorem{pro}[thm]{Proposition}
 
 \newtheorem{lem}[thm]{Lemma}
 
 \theoremstyle{definition}

 \theoremstyle{remark}
 \newtheorem{rem}{Remark}
 
 \newtheorem{exa}[thm]{Example}

\begin{document}
\title[Isometric actions on pseudo-Riemannian nilmanifolds]{Isometric actions on pseudo-Riemannian
  nilmanifolds}

\author{Viviana del Barco}
\email{delbarc@fceia.unr.edu.ar}

\address{Universidad Nacional de Rosario, ECEN-FCEIA, Depto de Matem\'a\-tica, Pellegrini 250, 2000 Rosario, Santa Fe, Argentina.}

\author{Gabriela P. Ovando}
\email{gabriela@fceia.unr.edu.ar}

\address{CONICET -  Universidad Nacional de Rosario, ECEN-FCEIA, Depto de Matem\'a\-tica, Pellegrini 250, 2000 Rosario, Santa Fe, Argentina.}

\date{\today}

\begin{abstract}  This work deals with the structure of the isometry group of 
pseudo-Rieman\-nian  2-step nilmanifolds. We study the action by isometries of several groups and we construct 
examples showing substantial differences with the Riemannian situation; for instance the action 
of the nilradical of the isometry group does not need to be transitive. For a nilpotent Lie group endowed with a 
 left-invariant pseudo-Riemannian metric  we study conditions for which the subgroup of isometries fixing the identity element 
 equals the subgroup of isometric automorphisms. This set equality  holds for pseudo-$H$-type Lie groups.
\end{abstract}

\thanks{{\it (2010) Mathematics Subject Classification}: 53C50, 53C30, 22E25, 53B30.
}

\thanks{ }

\maketitle

\section{Introduction}
A pseudo-Riemannian nilmanifold is a pseudo-Riemannian manifold $M$  which admits a transitive action
 by isometries of a nilpotent Lie group. The interest on 
these manifolds  has been  renovated in the last
 years motivated by their applications not only in mathematics but also in physics 
 (see for instance \cite{Du,Fi} and references therein). A typical question is the possibility of 
the extension to the pseudo-Riemannian case of several
 properties already known in  the Riemannian situation, a topic which  could be quite
 complicated. 
Flat pseudo-Riemannian nilmanifolds were investigated in \cite{DI}; for non-flat
 nilmanifolds there are today many open problems.  

  Let $M$ denote a  Riemannian manifold with isometry group $\Iso(M)$.  Wolf proved
in \cite{Wo}  that {\em if a  connected nilpotent Lie group $N\subseteq \Iso(M)$ acts transitively
 on  $M$ then $N$ is unique, it is the nilradical of the isometry group, and
 the transitive action of $N$ is also simple. Thus $M$ can be identified with the nilpotent Lie
 group $N$ equipped with a left-invariant metric.
 Furthermore the subgroup $H$ of isometries fixing the identity element coincides with the group $H^{aut_N}$ of isometric automorphisms of $N$ and therefore  the isometry group is the semidirect product $\Iso(M)=N\rtimes H$ (*)} (see \cite{Go, Wi}). 

Later Kaplan  \cite{Ka1} studied other  isometric actions on a family of 2-step nilmanifolds, namely on $H$-type Lie groups. 
It was shown that 
the group of isometric automorphisms coincides
 with the group of isometries of $N$ fixing its identity element and the distribution
\begin{equation}\label{dist}
 TN = \mv N \oplus \zz N.
\end{equation}
These subbundles are obtained by translation on the left of the splitting at the Lie algebra level
\begin{equation}\label{split}
 \mn= \mv \oplus \zz,
\end{equation}
 where $\mn$ is the Lie algebra of $N$, $\zz$ denotes its center
and $\mv$ the orthogonal complementary subspace  of $\zz$. 

Our goal here is to investigate  some Lie groups acting by isometries on a fixed pseudo-Riemannian 2-step nilpotent Lie group: the group of isometries preserving the splitting (\ref{dist}), 
the group of isometric automorphisms and the full isometry group. We get several 
 results concerning this topic, improving the results in \cite{CP} and we have examples showing difficulties. 
 As noticed recently by Wolf et al. \cite{WZ} the question of the structure of the nilradical 
of the isometry group for a pseudo-Riemannian nilmanifold is subtle. 

We exhibit a 2-step nilpotent Lie group $N$ equipped with a left-invariant 
Lorentzian metric such that:
\begin{itemize}
 \item the group  of isometric automorphisms is smaller than the subgroup of isometries fixing the identity
 element;
\item the Lie group $N$ is not normal into $\Iso(N)$, hence the algebraic structure (*) does not hold;
\item the action of the nilradical of $\Iso(N)$ is not transitive on $N$.
\end{itemize}

These facts reveal remarkable differences with the Riemannian si\-tuation. For Riemannian 
2-step nilmanifolds it is known that  every isometry is ''compatible'' with the splitting (\ref{split}). Geometrically the subspace $\mv$ correponds to negative eigenvalues of the Ricci operator while the subspace $\mz$ is described by the non-negative eigenvalues. We shall see that a similar characterization cannot be achieved for a metric with signature.

 Let $N$ denote a 2-step nilpotent Lie group equipped with a pseudo-Riemannian  left-invariant metric. If the center is non-degenerate one  gets a decomposition of the Lie algebra $\mn$ as in (\ref{split}). Under these hypothesis
 \begin{enumerate}
 \item[{\it a.}]  the group of isometric automorphisms coincides with the group 
of isometries fixing the identity element and preserving the splitting (\ref{dist});
\item[{\it b.}] we get  conditions to assert the equality 
 $H=H^{aut_N}$ so obtaining the structure (*) for the full isometry group.
\end{enumerate}
 In particular, the family of pseudo-$H$-type Lie groups satisfies the conditions in {\it b.} However 
this does not characterize this family.

 In the case of degenerate center the situation is more singular. We mainly work with 
bi-invariant metrics on 2-step nilpotent Lie groups, showing that {\it a.} above does not hold: there are isometric automorphisms not preser\-ving 
any kind of splitting (\ref{dist}). Even more, we present an example where there is no relationship between those groups. 

\section{Basic facts and notations} \label{sec1}

Let $(G, \bil)$ denote a Lie group equipped with a left-invariant pseudo-Rieman\-nian metric. In this section we recall some definitions and properties of groups  acting by  isometries. 

 Let $\Iso(G)$ denote the (full) isometry group of  $(G, \bil)$. This is a Lie group whenever it is equipped with the compact-open
topology. 
Since translations on the left are isometries, it is easy to verify that every $f\in \Iso(G)$ can be written as a product 
\begin{equation}
f= L_g\circ \varphi \label{comp}\end{equation}
  where $L_g$ denotes the translation by the element $g\in G$ and $\varphi$ is an isometry
 fixing the identity element. The subgroup of left-translations by elements of $G$ is  closed in $\Iso(G)$ and it is isomorphic
 to $G$. The subgroup of isometries fixing  the identity element denoted by $H$ is also a  closed subgroup of $\Iso(G)$ and due to (\ref{comp}) one has 
 $$\Iso(G)=G\cdot H.$$
 Let $\Aut(G)$ denote the group of automorphisms of $G$ and set $\Iso^{aut}(G)=G\cdot H^{aut_G}$ where $H^{aut_G}$ denotes the  group of isometric automorphisms of $G$, that is $H^{aut_G}=\Aut(G)\cap \Iso(G)$. Since for every automorphism $\phi\in \Aut(G)$ it holds $\phi \circ L_x = L_{\phi(x)} \circ \phi$, it follows that the subgroup of translations on the left is a  normal subgroup of the group $\Iso^{aut}(G)$,  thus one gets 
$$\Iso^{aut}(G)= G \rtimes H^{aut_G}.$$

A pseudo-Riemannian manifold is called {\em locally symmetric} if $\nabla R \equiv 0$, where $\nabla$ denotes 
the covariant derivative with respect to the Levi-Civita connection and $R$ denotes the curvature tensor. The Ambrose-Hicks-Cartan theorem (see for example \cite[Thm. 17, Ch. 8]{ON}) states that given a 
complete locally symmetric 
pseudo-Riemannian  manifold $M$, a linear isomorphism
$A : T_p M \to T_p M$ is the differential of some isometry of $M$ that fixes the point $p\in M$ if and 
only if it preserves the symmetric bilinear form that the metric induces into the tangent space and if for 
every $u,v, w \in T_p M$ the following equation holds:
\begin{equation} \label{ACH}
 R(Au, Av)Aw = A R(u, v)w.
\end{equation}


Let $\ggo$ denote the Lie algebra of $G$ which is identified with the Lie algebra of left-invariant vector fields on $G$.
Then for $G$ connected
  the following statements are equivalent (see \cite[Ch. 11]{ON}):

\begin{enumerate}
\item $\la\, ,\,\ra$  is right-invariant, hence bi-invariant;
\item $\la\,,\,\ra$  is $\Ad(G)$-invariant;
\item  the inversion map $g \to g^{-1}$ is an isometry of $G$;
\item $\la [u, v], w\ra + \la v, [u, w] \ra= 0$ for all $u,v,w \in \ggo$;
\item $\nabla_u w = \frac12 [u, w]$ for all $u,w \in \ggo$, where
$\nabla$ denotes the Levi Civita connection;
\item  the geodesics of $G$ starting at the identity element $e$ are the 
one parameter subgroups of $G$.
\end{enumerate}

 By 3. the pair $(G, \bil)$ is a pseudo-Riemannian symmetric space, that is, geodesic symmetries are isometries. 
 Usual computations show that 
   the curvature tensor is
 \begin{equation}
 R(u, w) = - \frac14 \ad([u,w]) \qquad \quad  \mbox{ for } u,w \in
 \ggo.
 \label{curvatura}
 \end{equation}
 

The following lemma is proved by applying the Ambrose-Hicks-Cartan theorem to the Lie group $G$ 
equipped with a bi-invariant metric and whose  curvature formula was given in  
(\ref{curvatura}) (see \cite{Mu}).
 
 \begin{lem} \label{iso} Let $G$ be a simply connected Lie group with a bi-invariant
 pseudo-Riemannian metric. Then a linear isomorphism $A : \ggo \to \ggo$ 
 is the differential of some isometry in $H$ if and only if for all $u,v,w\in \ggo$, the
 linear map $A$ satisfies the following two conditions:

\begin{enumerate}

\item $\la A u, A w \ra = \la u, w\ra $;

\item $ A[[u, v], w] = [[Au, Av], Aw]$.

\end{enumerate}
\end{lem}

Note. A symmetric bilinear form on a Lie algebra $\mgg$ satisfying 4. above is said to be ad-invariant. If it is non-degenerate we just call it a {\em metric}. 
\section{A homogeneous Lorentzian manifold of dimension four} \label{sec3}

In this section we study geometrical features of  a Lorentzian manifold of dimension
four and we show that it admits  a transitive and simple action by isometries of both  a solvable and  a 
nilpotent Lie group. \smallskip

Set $M$  the pseudo-Riemannian manifold $\R^4$ with the following Lorentz\-ian metric 
\begin{equation}\label{metric}
 g= dt (dz + \frac12 y dx - \frac12 x dy) + dx^2 + dy^2
\end{equation}
where $(t,x,y,z)$ are  usual coordinates  for $\RR^4$.
Denote  $v=(x, y)$ and for each  $(t_1, v_1, z_1)\in \RR^4$ consider the following 
differentiable functions of $M$:
\begin{equation}\label{actionN}
 L^N_{(t_1, v_1, z_1)}(t_2, v_2, z_2)=(t_1+t_2, v_1+v_2, z_1+z_2+\frac12 v_1^t J v_2)
\end{equation}
\begin{equation}\label{actionG}
 L^G_{(t_1, v_1, z_1)}(t_2, v_2, z_2)=(t_1+t_2, v_1+R(t_1)v_2, z_1+z_2+\frac12 v_1^t JR(t_1) v_2)
\end{equation}
where $J$ and $R(t)$ are the linear maps on $\RR^2$ given by
\begin{equation}\label{mat}
J=\left( \begin{matrix} 0 & 1 \\ -1 & 0 \end{matrix} \right), \qquad \qquad
R(t)=\left( \begin{matrix} \cos t & -\sin t \\ \sin t & \cos t \end{matrix} \right) \qquad t\in \RR.
\end{equation}
Both maps $L^N_{(t_1, v_1, z_1)}$ and $L^G_{(t_1, v_1, z_1)}$ are isometries of $(M, g)$: in 
fact on the basis $\{\deat, \deax, \deay, \deaz\}$ of $T \RR^4$ one has the following differentials
$$ L^N_{(t_1, x_1,y_1, z_1) *} = \left( \begin{matrix}
1 & 0 & 0 & 0 \\
0 & 1 & 0 & 0 \\
0 & 0 & 1 & 0\\
0 & -\frac12 y_1 & \frac12 x_1 & 1                    
                                   \end{matrix} \right)$$
$$ L^G_{(t_1, x_1, y_1, z_1) *} = \left( \begin{matrix}
1 & 0 & 0 & 0 \\
0 & \cos t_1 & -\sin t_1 & 0 \\
0 & \sin t_1 & \cos t_1 & 0\\
0 & \mu & \nu & 1                    
                                   \end{matrix} \right) \mbox{ with } 
\begin{array}l
\mu= \frac12( x_1 \sin t_1 - y_1 \cos t_1), \\ \\
\nu = \frac12(x_1 \cos t_1+ y_1 \sin t_1).
\end{array}$$
 Thus
 the maps above  give isometric left-actions  of a solvable Lie group $G$ and a nilpotent
 Lie group  $N$ on $(M,g)$.  

Both, the Lie group $G$ and $N$ are modelled on $\RR^4$ with its
canonical differentiable structure and with the multiplication map where the translations 
on the left are induced by the maps 
$L^G$ on $G$ and $L^N$ on $N$. 

It is not hard to see that the actions of both groups $G$ and $N$  on 
$M$ are simple and transitive so that the homogeneous Lorentzian manifold $M$ can be
 represented as $(G, g_G)$ and $(N, g_N)$ where $g_G$ and $g_N$ are both given by the same formula 
(\ref{metric}). Canonical computations show that the metric $g_N$ is left-invariant on $N$ and the
 metric $g_G$ is
bi-invariant on $G$, that is,
translations on the left and on the right on $G$ are isometries for $g_G$ (see the previous section).

\begin{rem} The exponential map from $\ggo$ to  $G$ is not surjective as one verifies with the formula given in Section 4.2 \cite{BOV1}, a fact previously proved in \cite{St}.  Thus $G$ is a solvable Lie group which belongs to the class of examples named in Example 3.4 \cite{WZ} for which there exists a pair of points that cannot be joined by an unbroken geodesic. 

The Lie group $G$  known as the oscillator group \cite{St} underlies the
 Nappi-Witten space \cite{NW}. Since the Lorentzian metric $g_G$ is bi-invariant on $G$, the homogeneous manifold $M$ is symmetric.
\end{rem}
\subsection{The solvmanifold model}
Making use of the model $(G, g_G)$ one obtains the isometry group of $(M,g)$. Actually as an application of   
Lemma \ref{iso}, the group $H$ of isometries of $G$ - and hence $M$ - fixing the element $e=(0,0,0)$ (the identity element of
 $G$) has a differential at $e$ with matrix of the form
$$\left( \begin{matrix}
 \varepsilon & 0 & 0 \\
  Jv & A & 0 \\
  -\frac12 ||v||^2 & -(Jv)^t A & \varepsilon
  \end{matrix}
  \right) \qquad \mbox{ with }\quad \varepsilon=\pm 1,\,A\in\Or(2),\;v\in\R^2$$
where $\Or(2)$ denotes the orthogonal group of $\RR^2$. Thus 
$$H\simeq (\{1,-1\}\times \Or(2))\ltimes \RR^2$$
 (see \cite{BOV1} for more details).  Notice that $H$ has four connected components and the connected component of the identity coincides with the group of inner automorphisms of $G$ $$H_0=\{\chi_g:G\lra G,\; \chi_g(x)=gxg^{-1}:\;g\in G\}\simeq \SO(2)\ltimes \R^2.$$ Explicitly for  $(t,v,z)\in\R^4$ and  $g=(t_0, v_0, z_0)$
\begin{eqnarray}\label{eq:ig}
\chi_g(t,v,z) & = & (t,v_0+R(t_0)v-R(t)v_0,\\
& & \hspace{1cm}z+\frac12 v_0^tJR(t_0)v-\frac12 v_0^tJR(t)v_0-\frac12(R(t_0)v)^tJR(t)v_0). \nonumber
\end{eqnarray}

The following diffeomorphisms
\begin{eqnarray}\label{eq:is}
\psi_1(t,v,z)&=&(-t,Sv,-z),\ \ \text{ where } S(x,y)=(-x,y)\nonumber
\\ 
\psi_2(t,v,z)&=&(-t,R(-t)v,-z),\ \ 
\label{eq:fi}\\ 
\psi_3(t,v,z)&=&\psi_1\circ \psi_2(t,v,z)=(t,SR(-t)v,z), \nonumber
\end{eqnarray}
constitute isometries of $M$ fixing the element $e$ and they belong to different 
connected components of $H$. The other three connected components of $H$ are 
$$ H_0\cdot \psi_1,\qquad H_0\cdot \psi_2 \qquad\text{and } \quad H_0\cdot \psi_3.$$
 Besides, the group $H^{aut_G}$ of isometric automorphisms of $G$ is not connected since it corresponds to
 the connected components  $H_0$ and $H_0\cdot \psi_1$, where $\psi_1$ is as in (\ref{eq:is}). Notice that $\chi_{(-\pi,0,0)}\circ\psi_2 $ is the group inversion of $G$.
 
 \begin{rem} \label{normal} Since $\psi_2 \circ L^G_{(t_1, v_1, z_1)} \circ \psi_2^{-1}$ is not a 
translation on the left, the subgroup of left-translations $G$ is not normal in $\Iso(M)$. \end{rem}
\smallskip

Note that  $G$ is normal into $\Iso_0(M)$, the connected component of the identity of $\Iso(M)$, hence
$$\Iso_0(M)=  G\rtimes H_0.$$

The Lie algebra of $\Iso(G)$ denoted by $\iso$ corresponds to the vector space spanned by 
$\{f_0, f_1, f_2, e_0, e_1, e_2, e_3\}$ obeying the non-trivial Lie bracket relations:
$$
\begin{array}{ccc}
{\begin{array}{rcl}
{[f_0, f_1]} & = & f_2 \\ {[f_0, e_2]}  & = & - e_1 \\ {[e_0, e_1]} & = & e_2 \end{array}}
& { \begin{array}{rcl}
{[f_0, f_2]} & = & -f_1 \\  {[f_1, e_2]} & = & e_3  \\  {[e_0, e_2]} & = & -e_1 \end{array}}
& {\begin{array}{rcl} 
{[f_0, e_1]} & = &  e_2 \\
 {[f_2, e_1]} & = & - e_3 \\
  {[e_1, e_2]}  & = & e_3.
\end{array}}
  \end{array}
$$

\subsection{The nilmanifold model} We study the structure of the isometry group of the nilmanifold $M$ with respect to the nilpotent Lie group $N$. 

The nilpotent Lie group $N$ corresponds to
 the Lie group $\RR \times \Hh_3$ 
where $\Hh_3$ denotes the Heisenberg Lie group of dimension three. Notice that for any $(t_1,v_1,z_1)\in\R^4$ it holds
\begin{equation}L^G_{(t_1, v_1, z_1)}=L^N_{(t_1, v_1, z_1)} \circ \chi_{(t_1,0,0)}=\chi_{(t_1,0,0)} 
\circ L^N_{(t_1, R(-t_1) v_1, z_1)},\label{lgln}\end{equation}
so that the Lie algebra 
$\mn$ viewed into $\iso$ is spanned by the vectors
\linebreak $f_0-e_0, e_1, e_2, e_3$ obeying the non-zero  Lie bracket relation $$[e_1, e_2]=e_3.$$

Since $\chi_{(t_1,0,0)}\in H^{aut_N} \subseteq H$ we have 
$$\Iso(M)=  N\cdot H \quad \mbox{with }H\simeq(\{1,-1\} \times \Or(2))\ltimes \RR^2,$$
but {\em $N$ is not a normal subgroup of $\Iso_0(M)$}.

For the left-invariant metric given in (\ref{metric})
a skew-symmetric derivation $D$ of $(\nn, \meti \, ,\, \metd)$ must preserve both subspaces $\zz$ and $\mv$ and 
following canonical computations one gets the non-trivial equalities
$$De_1=\eta e_2 \qquad D e_2= -\eta e_1, \quad \eta \in \RR.$$
So the connected component of the identity in the subgroup of isometric automorphisms of $N$ is given by
$$H^{aut_N}_0=\{\chi_{(s,0,0)}: s\in\R\}. $$
Recall from (\ref{eq:ig}) that $ \chi_{(s,0,0)}(t,v,z)=(t, R(s)v, z)$  with $ R(s)$ as in (\ref{mat}).
Furthermore $\Iso^{aut}(N)$ is not connected. In fact   the map $\psi_1$ defined in (\ref{eq:is}) 
is also an isometric automorphism of $N$. 

We already proved that   $$\Iso^{aut}(N) \subsetneq\Iso(M).$$ Compare with \cite{CP}.

The nilradical of $\iso$ is the ideal of dimension five spanned by $\{f_1, f_2,$ $e_1, e_2, e_3\}$
 which does not
contain the subalgebra $\nn$. In fact, $\mn$ is not contained in the commutator $[\iso, \iso]$.
 In terms of the isometry functions the maximal connected normal 
nilpotent subgroup of $\Iso_0(M)$ corresponds to
$$ \tilde{N}=\{ L^G_{(0,w, z)} \circ \chi_{(0,v,0)}: \quad v,w, \in \RR^2, z \in \RR \}$$
and usual computations show that the orbit of a  point $(t_0,v_0,z_0)$ is given by
$$\mathcal O(t_0,v_0,z_0)=\{(t,v,z)\in M  : \quad  t=t_0\}.$$
This proves that the action of $\tilde{N}$ {\em is not transitive on the nilmanifold} $M$.

\medskip

\noindent We summarize the results for $(M, g)$.
\begin{itemize}
 \item $N$ acts transitively on $M$ but it is not a normal subgroup of $\Iso_0(M)$.
\item The maximal connected normal nilpotent Lie subgroup of $\Iso(M)$, namely the nilradical, 
 does not contain $N$ and its action on $M$ is not transitive.
\item $\Iso^{aut}(N)\subsetneq \Iso(M)$, where  $\Iso^{aut}(N)=N\rtimes H^{aut_N}$, $\Iso(M)= N \cdot H$.
\item $H^{aut_N}$ is not connected.

\end{itemize}

Compare this with Section 4.2 in \cite{Wo}, and  results in  \cite{CP}.

\section{Isometric actions for non-degenerate center}
This section concerns the study of isometric actions. Let $(N, \bil)$ denote a simply connected 
2-step nilpotent Lie group equipped with a left-invariant
pseudo-Riemannian metric and let $\bil$ denote the metric on its corresponding Lie
algebra $\mn$. 

Let $\mz$  be the center of $\mn$ and assume the restriction of $\bil$ to $\mz$ is non-degenerate. Hence  there exists an orthogonal decomposition of $\mn$
\begin{equation}
\label{ortdec} \mn=\mv\oplus \mz
\end{equation}
so that  $\mv$ is also non-degenerate. This induces on the Lie group $N$ left-invariant
orthogonal distributions $\mv N$ and $\mz N$ such that $TN=\mv  N\oplus\mz  N$. 

Denote by $\Iso^{spl}(N)$ the group of isometries of $N$ that preserve the
splitting $TN=\mv N \oplus \mz N$ \cite{CP,Ka1}. Notice that translations on the left by elements of the group $N$  preserve
this splitting. Thus
 $$\Iso^{spl}(N) = N\cdot H^{spl}$$ where $H^{spl}$ is the subgroup of isometries
which preserve the splitting and fix the identity element of $N$.  When the metric is positive
definite, one has  \cite{Eb,Ka1}:
\begin{equation} \label{isomgroup}
 \Iso(N)=\Iso^{aut}(N)=\Iso^{spl}(N).
\end{equation}

The purpose here is to analyze the group equalities above in the pseudo-Riemannian case 
 and occasionally to state new relationships between these three groups. 

Since the pseudo-Riemannian metric on $N$ is invariant by left-translations we study the geometry
 of $N$  as effect from $(\nn, \bil)$. 

Given $u\in\mn$, denote by $\ad_u^*$ the adjoint transformation of $\ad_u$ with
respect to $\bil$. One verifies that when $u\in \mv$ and $w\in \mz$ it holds $\ad_u^*w\in\mv,$
 while $\ad_u^*w=0$ if $u\in\mz$ or $u,w\in\mv$. Furthermore each $w\in\mz$  defines a linear transformation
$j(w):\mv\lra\mv$ by
$$j(w)u=\ad_u^*w \mbox{ for all } u\in\mv,$$ 
so that 
\begin{equation}\label{eq:jota}
\lela j(w)u,u'\rira =\lela w,[u,u']\rira \quad \mbox{ for all } u,u'\in\mv.
\end{equation}
Thus for $w\in \zz$, the map $j(w)$ belongs to $\sso(\mv)$,   the Lie algebra of skew-symmetric maps of $\mv$ with respect to $\bil$ and  one gets that $j: \mz \to \sso(\mv)$ is a linear homomorphism. As in the Riemannian case, the maps $j(w)$
capture important 
geometric information of the pseudo-Riemannian space $(N,\bil)$. 

The  covariant derivative $\nabla$ 
relative to the Levi Civita connection of ($N$, $\bil$) evaluated on left-invariant
vector fields is
$$2\,\nabla_uw=[u,w]-\ad_u^*w-\ad^*_wu,\quad u,w\in\mn, $$ 
which together with   the formula for $j:\mz\lra \mathfrak{so}(\mv)$ in (\ref{eq:jota}) gives
\begin{equation}\label{eq:nabla}\left\{
\begin{array}{ll}
\nabla_u w=\frac12 \,[u,w] & \mbox{ if } u,w\in\mv,\\
\nabla_u w=\nabla_wu=-\frac12 j(w)u & \mbox{ if } u\in\mv,\,w\in\mz,\\
\nabla_u u'=0& \mbox{ if } u, u'\in\mz.
\end{array}\right.
\end{equation}

Since for simply connected nilpotent Lie groups the exponential map $\exp:\mn\lra N$ is a diffeomorphism,  it is possible to
define smooth maps $b:N\lra \mv$ and $a:N\lra \mz$ such that for a given $n\in N$ one writes
\begin{equation}
n=\exp(b(n)+a(n))\label{exp}.\end{equation} 
Let 
$\{b_1,\ldots,b_m\}$ be a basis of $\mv$ and $\{a_1,\ldots,a_p\}$ be a basis of $\mz$, then 
there are defined
maps $\{\beta_i,\alpha_j:N\lra \R:\,i=1,\ldots,m,\,j=1,\ldots,p\}$ for which
$$b(n)=\sum_{i=1}^m\beta_i(n)b_i, \qquad a(n)=\sum_{j=1}^p \alpha_j(n)a_j . $$
Thus $\varphi=(\beta_1,\ldots,\beta_m,\alpha_1,\ldots,\alpha_p)$ is a global coordinate system 
for $N$ where at $n\in N$ it holds
\begin{equation}\label{vectores}
\begin{array}{rcl}
\frac{\partial}{\partial \beta_i}_{|_n}&=&L_{n_*}{|_e}(b_i+\frac12
\sum_{k=1}^m\,[b_i,\beta_k(n)b_k]),\\
\frac{\partial}{\partial \alpha_j}_{|_n}&=&L_{n_*}{|_e}(a_j).
\end{array}
\end{equation}
To verify these equalities see formulas for the exponential map in \cite{Eb}.

Let $\gamma:I\lra N$ be a curve on $N$ and write $b$ and $a$ for the vector
valued maps 
$\gamma(t)=\exp(b(t)+a(t))$. Making use of  the equalities in (\ref{vectores}) one gets  
\begin{eqnarray}
\frac{d}{dt}\gamma(t)&=&\sum_{i=1}^m \betaidot \frac{\partial}{\partial
\beta_i}_{|{\gamma(t)}} +\sum_{i=j}^p \alphajdot \frac{\partial}{\partial
\alpha_j}_{|{\gamma(t)}}\nonumber\\
&=&L_{\gamma(t)_*}{|_e}\left( \sum_{i=1}^m \betaidot x_i+ \sum_{j=1}^p
(\sum_{k=1}^m\frac12\, c^j_{ik}\betaidot v_k+\alphajdot)z_j\right)\nonumber\\
&=&L_{\gamma(t)_*}{|_e}\left(\stackrel{\cdot}{b}
+\stackrel{\cdot}{a}+\frac12\, [\stackrel{\cdot}{b},b]\right),\nonumber
\end{eqnarray}
where $c_{ik}^j$ denote the structure constants for the Lie algebra. Notice that $L_{\gamma(t)_*}{|_e}(\stackrel{\cdot}{b})$ and
$L_{\gamma(t)_*}{|_e}(\stackrel{\cdot}{a}+\frac12\, [\stackrel{\cdot}{b},b])$ are
the components of $d\gamma/dt$ in the bundles $\mv N $ and $ \mz N$ respectively.


So the covariant derivative of $\stackrel{\cdot}{\gamma}$ along $\gamma$ is given by 
\begin{eqnarray}\nonumber
\nabla_{d\gamma/dt}d\gamma/dt&=&L_{\gamma(t)_*}{|_e}\left(
\nabla_{\stackrel{\cdot}{b} + \stackrel{\cdot}{a}+\frac12
[\stackrel{\cdot}{b},b] }\stackrel{\cdot}{b} +(\stackrel{\cdot}{a}+\frac12
[\stackrel{\cdot}{b},b])\right).
\end{eqnarray}
Denote by $\sigma=\,\stackrel{\cdot}{b} +\stackrel{\cdot}{a}+\frac12
[\stackrel{\cdot}{b},b]$ the curve in $\mn$. Then
\begin{equation}\nabla_{\sigma(t)}\stackrel{\cdot}{b}\,=
\nabla_{\sigma(t)}(\sum_{i=1}^m\betaidot b_i)=\sum_{i=1}^m\betaidotdot
b_i+\sum_{i=1}^m\betaidot \nabla_{\sigma(t)}b_i\nonumber
\end{equation} 
which after (\ref{eq:nabla}) equals
\begin{equation}
\nabla_{\sigma(t)}\stackrel{\cdot}{b}\,=\,\stackrel{\cdot\cdot}{b}+\sum_{i=1}^m\betaidot
\left( \frac12 [\stackrel{\cdot}{b},b_i]-\frac12
j(\stackrel{\cdot}{a}+\frac12[\stackrel{\cdot}{b},b])b_i
\right)\,=\,\stackrel{\cdot\cdot}{b}-\frac12
j(\stackrel{\cdot}{a}+\frac12[\stackrel{\cdot}{b},b])\stackrel{\cdot}{b}.\label{nablaxpunto}\end{equation}

Similar computations give 
\begin{equation}
\nabla_{\sigma(t)}\left(\stackrel{\cdot}{a}+\frac12
[\stackrel{\cdot}{b},b]\right)\,=\,
\,\stackrel{\cdot\cdot}{a}+\frac12 [\stackrel{\cdot\cdot}{b},b]-\frac12
j(\stackrel{\cdot}{a}+\frac12[\stackrel{\cdot}{b},b])\stackrel{\cdot}{b}.\nonumber
\end{equation}
Therefore a curve $\gamma:I \to N$ is a geodesic if and only if the curve  $\sigma:I \to \mn$ satisfies 
$\nabla_{\sigma(t)}\sigma(t)=0$, that is,  the following system of equations is satisfied
\begin{equation}\left\{\begin{array}{l}
\stackrel{\cdot\cdot}{b}-j(\stackrel{\cdot}{a}_0
){\stackrel{\cdot}{b}}=0\vspace{0.1cm}\\
\stackrel{\cdot}{a}+\frac12 [\stackrel{\cdot}{b},b]=\stackrel{\cdot}{a}_0
\end{array}\right.\quad \,\mbox{ where } \stackrel{\cdot}{a}_0 \mbox{ denotes
}\stackrel{\cdot}{a}(0). \label{geodesic}
\end{equation}
 
Let $f$ be an automorphism of $N$. Its differential at the identity $e$ satisfies
 $f_*(\mz) \subseteq \mz$ and if moreover  $f$ is an isometry its differential  also preserves the orthogonal complement $\mz^\bot=\mv$: 
$f_*( \mv )\subseteq \mv$. In view of $f \circ L_n = L_{f(n)} \circ f$ for every $n\in N$,  $f$ preserves the invariant distributions $\mz N$ and $\mv N$ whenever it is an isometric automorphism. Therefore
\begin{equation}\label{autis}
\Iso^{aut}(N) \subseteq \Iso^{spl}(N).
\end{equation}

\begin{pro} \label{autspl}
Let $N$ be a simply-connected 2-step nilpotent Lie group with a left-invariant pseudo-Riemannian
 metric such that  its center  is non-degenerate. Then
$$\Iso^{spl}(N)=
\Iso^{aut}(N).$$
\end{pro}

\begin{proof}
In view of (\ref{autis}) we should prove that every isometry  $f\in H^{spl}$ is an automorphism of $N$. 
The  proof follows from the next equality for $f$, we shall prove:
$$f_*(j(u)w)=j(f_*u)f_*w \qquad \mbox{ for all }w\in \mv, u\in
\mz. $$
 Let $\gamma=\exp(b(t)+a(t))$ denote  the geodesic throughout $e$ such that \linebreak 
$\stackrel{\cdot}{\gamma}(0)=w+u$, that is $\stackrel{\cdot}{b}(0)=w$ and
$\stackrel{\cdot}{a}(0)=u$. 
Since $f_*$ is an isometry preserving the splitting the geodesic
$\tilde{\gamma}=f\circ \gamma$ can be written as $\tilde{\gamma}(t)=\exp
(f_*b(t)+f_*a(t))$. The Equations (\ref{nablaxpunto}) and  (\ref{geodesic}) at $t=0$   for both
geodesics $\gamma$ and $\tilde{\gamma}$ give
\begin{eqnarray}
j(f_*u)f_*w&=&\left( j(f_*\stackrel{\cdot}{a}_0)
f_*\stackrel{\cdot}{b}\right)_{t=0}\nonumber
\\
&=&2\left(\nabla_{f_*(\stackrel{\cdot}{b} +(\stackrel{\cdot}{a}+\frac12
[\stackrel{\cdot}{b},b]))}f_*\stackrel{\cdot}{b}\right)_{t=0}\nonumber\\
&=&2\,f_*\left(\nabla_{\stackrel{\cdot}{b} +(\stackrel{\cdot}{a}+\frac12
[\stackrel{\cdot}{b},b])}\stackrel{\cdot}{b}\right)_{t=0}\nonumber\\
&=&f_*(j(u)w), \nonumber
\end{eqnarray}
as intended to prove. 

Now consider $w_1,w_2\in\mv$. For any $u\in\mz$ it holds
\begin{eqnarray}
\lela f_*[w_1,w_2],u\rira &=&\lela [w_1,w_2],f^{-1}_*u\rira\;=\;\lela
w_2\;,\;j(f^{-1}_*u)w_1\rira\;=\nonumber\\
&=&\lela w_2\;,\;j(f^{-1}_*u)(f^{-1}_*(f_*w_1))\rira\;=\;\lela
w_2\;,\;f^{-1}_*(j(u)f_*w_1)\rira\;=\nonumber\\
&=&\lela f_*w_2\;,\;j(u)f_*w_1\rira\;=\;\lela [f_*w_1,f_*w_2]\,,\,u\rira ,\nonumber
\end{eqnarray}
which together with the fact that $\mz$ is non-degenerate implies that $f_*$ is a
Lie algebra automorphism. Consequently $f\in H^{aut_N}$.
\end{proof}

The proof above extends to the pseudo-Riemannian setting that one performed by Kaplan  in \cite{Ka1}. 
 Below we investigate geometrical properties of pseudo-Riemannian 2-step nilpotent Lie groups
to get conditions to assert that 
 the group of isometries of $N$ preserving the splitting coincides
 with the full isometry group of $N$.


The Ricci tensor of $(N, \bil)$ can be seen at the Lie algebra level as  the bilinear form on $\mn$ defined throughout the curvature tensor $R$ by  $$\Ric(u,w)=\operatorname{trace}(\xi\lra
R(\xi,u)w)\quad\mbox{ for }\quad u,w\in \mn.$$
 Since $\Ric$ is a symmetric form on $\mn$ there exists a linear operator
$\Rc:\mn\lra \mn$ such that 
 $$\lela \Rc u,w\rira=\Ric(u,w)$$
  which is called the {\em Ricci operator}. The {\em scalar curvature} $s$ of $N$  is the
 trace of the Ricci operator $\Rc$.

In the pseudo-Riemannian case the formulas for the Ricci operator are slightly
different from those in the Riemannian case (see \cite{Eb}). Recall that a basis
$\{w_1,\ldots,w_n\}$ of $\mn$ is said to be {\em orthonormal} if $\lela w_i,w_j\rira
=\pm \delta_{ij}$. The proof of the next proposition follows from usual computations and it can be seen in \cite{Ov}.
\begin{pro}\label{pro1} Let $(N, \bil)$ denote a 2-step nilpotent Lie group equipped
with a left-invariant pseudo-Riemannian metric such that the center is
non-degenerate.
\begin{enumerate}
\item The Ricci operator leaves $\mv$ and $\mz$ invariant.
\item If $\{a_1,\ldots,a_p\}$ is an orthonormal basis of $\mz$ then
\begin{equation}\label{eq:tv}
\Rc|_{\mv}=\frac12 \sum_{i=1}^p\varepsilon_k j(a_k)^2 \qquad \mbox{ with }
\varepsilon_k=\lela a_k,a_k\rira.
\end{equation}
\item $\Ric(u,u')=-\frac{1}{4} \operatorname{trace}(j(u)j(u'))$, for all $u, u'\in \mz$.
\end{enumerate}
\end{pro}

We proceed with the study of the eigenvalues of the Ricci operator $\Rc$. 
  Recall that if $U$ is a real vector space  its complexification is the  vector space
 $$U^\CC=U\otimes_\RR \CC$$  and such that $\dim_\RR U=\dim_\CC U^\CC$. 
A real linear transformation $T$ of $U$ defines a $\CC$-linear operator on $U^\CC$ as 
$T(u\otimes z)=T(u)\otimes z$ for all $u\in U$. In addition, if $U=U_1\oplus U_2$ 
then $U^\CC= U_1^\CC\oplus U_2^\CC$ and  $U_i$ is invariant under $T$ if and only if
 $U_i^\CC$ is invariant under the complex transformation $T$.

Denote with $\lambda_1,\lambda_2,\cdots,\lambda_s$ the different eigenvalues of the (complex)
 Ricci operator 
$\Rc$. Since the metric is left-invariant,  the eigenvalues of $\Rc$ are 
constant on $N$. 
The subspace of $\mn^\CC$ associated to the eigenvalue $\lambda_i$ is 
\begin{equation}\label{Vlambda} V_{\lambda_i}=\ker (\Rc-\lambda_i I)^{r_i},\end{equation}
where $r_i$ is the degree of $\lambda_i$ in the minimal polynomial of $\Rc$. The 
Jordan decomposition theorem states that 
\begin{equation}\label{decn}
\mn^\CC =V_{\lambda_1}\oplus V_{\lambda_2}\oplus\cdots \oplus V_{\lambda_s}.\end{equation}
Translating on the left the spaces  above, at a generic point $n\in N$ one obtains 
\begin{equation}
\label{tnc}T_nN^\CC= L_{n\,*}|_e V_{\lambda_1}\oplus L_{n\,*}|_e V_{\lambda_2}\oplus\cdots \oplus
 L_{n\,*}|_e V_{\lambda_s}.
\end{equation}
The subspace $L_{n\,*}|_e V_{\lambda_i}$ of $T_nN^\CC$ is that one corresponding to 
the eigenvalue $\lambda_i$ of the Ricci tensor at $n$, $\Rc_n$; that is, 
$$L_{n\,*}|_e V_{\lambda_i}=\ker(\Rc_n-\lambda_i I)^{r_i}.$$

Given an isometry $f$ of $N$, it holds 
\begin{equation}f_*|_n \,\Rc_n=\Rc_{f(n)}\,f_*|_n \quad\mbox{for all} 
\quad n\in N,\label{difatp}\end{equation}
 and this formula is also valid for the corresponding complexified linear transformations. 
The last two equations  yield to
\begin{eqnarray}\label{a}
u\in L_{n\,*}|_e V_{\lambda_i} &\Leftrightarrow& (\Rc_n-\lambda_{i} I)^{r_{i}} u=0\nonumber\\
&\Leftrightarrow& f_*|_n((\Rc_n-\lambda_{i} I)^{r_{i}} u)=0\nonumber\\
&\Leftrightarrow& (\Rc_{f(n)}-\lambda_{i} I)^{r_{i}} (f_*|_n u)=0\nonumber\\
&\Leftrightarrow& f_*|_n( u) \in L_{f(n)\,*}|_e  V_{\lambda_i} .\end{eqnarray}
Therefore the direct sum of vector spaces in (\ref{tnc}) is preserved by isometries. 

\begin{lem} \label{lm2} Let $(N,\bil)$ be a 2-step nilpotent Lie group such that $\bil$ is a pseudo-Riemannian left-invariant  metric for which the center is non-degene\-rate. Assume 
\begin{eqnarray}\label{condic}
\mv^\CC&=&V_{\lambda_1}\oplus\cdots\oplus V_{\lambda_j},\nonumber\\
\mz^\CC&=&V_{\lambda_{j+1}}\oplus\cdots\oplus V_{\lambda_s},
\end{eqnarray}
for  the different eigenvalues $\lambda_1,\lambda_2,\cdots,\lambda_s$ of the Ricci operator $\Rc$ with
 $V_{\lambda_i}$ the eigenspace  corresponding to $\lambda_i$. Then every isometry of $N$
 preserves the splitting $TN=\mv N\oplus \mz N$, that is,
$$\Iso(N)=\Iso^{spl}(N).$$
\end{lem}

\begin{proof}
The hypothesis in (\ref{condic}) implies that for any $n\in N$, 
\begin{eqnarray}
\mv N_n^\CC&=& L_{n\,*}|_eV_{\lambda_1}\oplus\cdots\oplus L_{n\,*}|_e  V_{\lambda_j}, \nonumber\\
\mz N_n^\CC&=&L_{n\,*}|_e V_{\lambda_{j+1}}\oplus\cdots\oplus L_{n\,*}|_e V_{\lambda_s}.\nonumber\end{eqnarray}

Let $f$ be an isometry of $N$, then $f_*|_n L_{n\,*}|_e V_{\lambda_i}= 
L_{f (n)\,*}|_e V_{\lambda_i}$ as a consequence of the conditions in (\ref{a}). Therefore $$f_*|_n(\mv N_n^\CC)=\mv N_{f(n)}^\CC \quad
 \mbox{ and }\quad f_*|_n(\mz N_n^\CC)=\mz N_{f(n)}^\CC $$
from which we conclude $f_*|_n (\mv N_n)=\mv N_{f(n)}$, $ f_*|_n (\mz N_n)=\mz N_{f(n)}$ and  so 
$f\in \Iso^{spl}(N)$.
\end{proof}

A large family of nilpotent Lie algebras satisfying the hypothesis of the lemma above is the 
family of pseudo-$H$-type  Lie algebras. \footnote{We find this name for the first time in \cite{Ci}, see also \cite{JPP}.}

A nilpotent Lie algebra $\mn$ (or its corresponding simply connected  Lie group) equipped with a
metric $\bil$ for which the center is non-degenerate is said to be {\em
of pseudo-$H$-type} whenever it satisfies
\begin{equation}\label{tipoH} 
j(u)^2=-\lela u,u\rira I\qquad \mbox{ for all }u\in\mz.
\end{equation}
Positive definite metric Lie algebras satisfying (\ref{tipoH}) are already
known as $H$-type Lie algebras, introduced by Kaplan in \cite{Ka1}.  $H$-type Lie groups are 
2-step nilpotent and they are natural generalizations
of the Iwasawa $N$-groups associated to semisimple Lie groups of real rank one.

 Notice that pseudo-$H$-type Lie algebras are not necessarily
non-singular since  vectors of zero norm could satisfy (\ref{tipoH}).
For any nilpotent Lie group of pseudo-$H$-type the three groups  in
(\ref{isomgroup}) coincide. 

\begin{thm} \label{teo1} Let $(N,\bil)$ denote a  pseudo-$H$-type  Lie group. Then 
\begin{enumerate}
\item $\Iso^{aut}(N)=\Iso^{spl}(N)=\Iso(N).$
\item the scalar curvature of $(N,\bil)$ is negative.
\end{enumerate}
\end{thm}

\begin{proof} The first equality in 1. holds after Proposition \ref{autspl}. We use the previous lemma to prove the second one. Indeed we show that for pseudo-$H$-type Lie algebras, the Ricci operator is diagonalizable and negative (resp. positive) definite on $\mv$ (resp. on $\mz$).
 
 Let $\{a_k\}_{k=1}^p$ be an orthonormal basis of $\mz$. The fact of $N$ being of
pseudo-$H$-type implies $j(a_k)^2=-\lela a_k,a_k\rira I_m=-\varepsilon_k I_m$ with $\varepsilon_k=\pm 1$
 for 
$k=1,\ldots, p$ and $m=\dim \mv$. Then, according to (\ref{eq:tv}) the Ricci operator satisfies
\begin{equation}
\Rc|_{\mv}=\frac12\sum_{k=1}^p\varepsilon_k j(a_k)^2= -\frac12 \sum_{k=1}^p\varepsilon_k^2
I_m=-\frac{p}{2} I_m
\end{equation}
so $\Rc$ is negative definite on $\mv$.

On the other hand for $u\in \mz$
$$\Ric(u,u)=-\frac14 \operatorname{trace} (j(u)^2)=\frac14
\operatorname{trace} (\lela u,u\rira I_m)=\frac{\lela u,u \rira}{4}\,m.$$
 Hence $\lela \Rc u,u\rira =\Ric(u,u)=\lela \frac{m}{4} u,u \rira$ 
for all $u\in\mz$.
Polarizing this identity one gets $\lela \Rc u,u'\rira=\lela \frac{m}{4} u,u' \rira$
for any $u,u'\in\mz$ and therefore $\Rc=\frac{m}4\,I_p$ on $\mz$. In particular $\Rc$ is positive definite on $\mz$.

Clearly, the eigenvalues of $\Rc$ are $\lambda_1=-p/2$ and $\lambda_2=m/4$ and the subspace
 $V_{\lambda_i}$ in (\ref{Vlambda}) is the eigenspace corresponding to $\lambda_i$, 
for each $i=1,2$. Moreover $\mv=V_{\lambda_1}$ and $\mz=V_{\lambda_2}$, so requirements 
(\ref{condic}) are satisfied and the first assertion of the theorem follows.

From the proof above, the trace of the Ricci operator 
 is  $s=-pm/4$ where $p=\dim \mz$ and $m=\dim \mv$, hence negative and 2. follows. 
\end{proof}


A natural question at this point is the validity of the converse of the previous result. 
The next example gives a negative answer, that is pseudo-$H$-type Lie groups are not the only ones for which $\Iso^{aut}(N)=\Iso(N)$.
 
 Recall that solvable Lie groups endowed with a Riemannian left-invariant me\-tric have non-positive 
scalar curvature (see Theorem 6 \cite{Je}). As the next example shows, this assertion does not hold in the indefinite case.
 

\begin{exa}\label{exa1} Let  $\Hh_3$ be the Heisenberg Lie group  and denote with $\hh_3$ its Lie algebra which  is spanned by vectors $e_1,e_2,e_3$ satisfying the non-zero Lie bracket relation 
$[e_1,e_2]=e_3$. Consider the left-invariant metric on $\Hh_3$ induced by the metric on 
$\hh_3$ given by $$-\lela e_1,e_1\rira=\lela e_2,e_2\rira=\lela e_3,e_3\rira=1; $$ 
in particular the center of $\hh_3$ is non-degenerate. After the computation of 
$j(e_3)$ it holds $j(e_3)^2= I$ implying that  $\hh_3$ is not a pseudo-$H$-type Lie algebra. 

By Proposition \ref{pro1}, the Ricci operator satisfies  $\Rc|_\mv=\frac12 I$ and $\Rc(e_3)=-\frac12 e_3$. So $\Rc$ is diagonalizable and the scalar curvature of $(\Hh_3,\bil)$ is $s=1/2$. Also $\mv=V_{1/2}$ and $\mz=V_{-1/2}$, thus Lemma \ref{lm2}  
  leads us to $$\Iso(\Hh_3)=\Iso^{aut}(\Hh_3). $$ 
In \cite{BOV1} it is shown that this non-flat pseudo-Riemannian nilpotent Lie group
 satisfies $\Iso^{aut}(\Hh_3)=\Hh_3\rtimes\Or(1,1)$. 
 \end{exa}

\begin{rem} The Lorentzian nilpotent Lie group $(N, g_N)$ in the previous section  shows that in the pseudo-Riemannian case the isometries do not need to preserve the splitting $\vv N \oplus \zz N$, even when the center is non-degenerate (compare with \cite{CP}).
Using Proposition \ref{pro1} one gets that the Ricci tensor of this manifold is non-zero but nilpotent. Then the Lie algebra $\mn$ of $N$ is the subspace defined in (\ref{Vlambda}) corresponding to the zero eigenvalue, that is $V_0=\mn$. Since $V_0$ has non-trivial intersections with $\mv$ and $\mz$, a decomposition as in (\ref{condic}) is not possible.

The fact that the Ricci operator is nilpotent implies that the scalar curvature of the Lie groups $(G,g_G)$ and $(N,g_N)$ are zero but they are not Ricci flat. Note that $G$ is 3-step solvable (see Theorem 7 in \cite{Je}).
\end{rem}

Once one knows that for $(N, \bil)$ it holds $H^{aut_N}=H$, the isometry group can be computed as follows. Recall that since $N$
  is simply connected,  we do not distinguish between the group of automorphisms of $N$ and of $\mn$. The isotropy group $H$ is
given by
\begin{equation}\label{auton}
H = \{(\phi, T) \in \Or(\mz, \bil_{\mz}) \times \Or(\mv, \bil_{\mv}) : Tj(w)T^{-1} = j(\phi w), \quad w \in \mz\}
\end{equation}
where $\Or(V, \bil)$ denotes the group of isometric linear maps of $(V, \bil)$. The Lie algebra of $G$ is given by 
\begin{equation}\label{deron}
\hh = \{(A,B) \in \sso(\mz, \bil_{\mz}) \times \sso(\mv, \bil_{\mv}) : [B, j(w)] = j(Aw), \quad w \in \mz\}
\end{equation}
where $\sso(V, \bil)$ denotes the set of skew-symmetric linear maps of $(V, \bil)$. 
In fact, let  $\psi$ denote a orthogonal automorphism of $(\mn, \bil)$. Thus  $\psi(\mz)\subseteq \mz$ and since $\mv = \mz^{\perp}$ then  $\psi(\mv)\subseteq \mv$. Set $\phi := \psi_{|_{\mz}}$ and $T :=  \psi_{|_{\mv}}$, thus $(\phi, T) \in
\Or(\mz, \bil_{\mz}) \times \Or(\mv, \bil_{\mv})$ such that
$$\begin{array}{rcl}
\lela \phi^{-1}[u, v]; x\rira & = &\lela [Tu, Tv]; j(x)\rira \quad \mbox{ if and only if }\\
\lela j(\phi x)u, v \rira  & = & \lela j(x)Tu, Tv\rira
\end{array}$$
which implies (\ref{auton}). By derivating (\ref{auton}) one gets (\ref{deron}).
The following proposition is the correct version of that in \cite{Ov}:

\begin{pro} Let $N$ denote a simply connected 2-step nilpotent Lie group
endowed with a left-invariant pseudo-Riemannian metric, with respect to which
the center is non-degenerate and such that $H=H^{aut_N}$. Then the group of isometries is 
$$\Iso(N)=N \rtimes H^{aut_N} $$
where $N$ acts as the group of left-translations by elements of $N$ and the isotropy subgroup $H$ is given by  (\ref{auton}) with Lie algebra as in (\ref{deron}).
\end{pro}


\section{Isometric actions for degenerate center}


Let $N$ denote a simply connected nilpotent Lie group with Lie algebra $\mn$ and set a direct sum of vector spaces
$$\mn = \mv \oplus \mz.$$
Note that $\mv$ does not necessarily coincide with $\mz^{\perp}$. By translations on the left, set $TN=\mv N \oplus \mz N$, a decomposition of left-invariant distributions. Define $\Iso^{spl}(N)$ as the group 
of isometries preserving this splitting. 

Assume $(N, \bil)$ is endowed with a bi-invariant metric (in this case $\mz^{\perp} \subseteq \mz$). Here we show that  the 
three  groups $\Iso(N)$, $\Iso^{aut}(N)$ and $\Iso^{spl}(N)$  could be very different and moreover there could be no relationship between $\Iso^{aut}(N)$ and $\Iso^{spl}(N)$.

\begin{lem}\label{lm1}
Let $(N, \bil)$ be a 2-step nilpotent Lie group equipped with a bi-invariant metric, let $\mn$ denote its Lie algebra. Then for any direct sum decomposition $\mn=\mz\oplus \mv$ the
automorphism $Ad(n)$ does not preserve $\mv$, whenever $n$ is non-central.
\end{lem}
\begin{proof}
Since the Lie group is 2-step nilpotent,  for any $n\in N$ there exists $w\in \mn$ such that $n=\exp\,(w)$  and $Ad(n)=I+\frac12\, \ad_w$. Consider a direct sum as vector spaces of the form $\mn=\mv\oplus\mz$. Suppose $n$ is non-central then $n=\exp (w)$ with  $w\notin \mz$ and let  $u\in \mn$ be such that $[w,u]\neq 0$.
Write $u=b+ a$ with $a\in \mz$ and $b\in\mv-\{0\}$.
Then $Ad(n)(b)=b+\frac12 [w,b]=b+\frac12 [w,u]$ having non-zero component on
$\mz$.
\end{proof}

\begin{rem}
Let $N$ denote a  pseudo-Riemannian nilpotent Lie group admitting  a lattice $\Gamma$ such that $\Gamma \backslash N$ is compact. Let $\Iso_0^{spl}(\Gamma \backslash N)$ denote the connected component of the subgroup of isometries preserving a fixed  splitting of $T(\Gamma \backslash N)$. In \cite{CP2} the authors conclude that $\Iso_0^{spl}(\Gamma \backslash N)\simeq T^m$ an $m$-dimensional torus. Its proof  should make use of the fact that there are no non-trivial inner automorphisms in this  group (Lemma \ref{lm1}).
\end{rem}

The fact that $N$ is endowed with a bi-invariant metric implies that the conjugation map $\chi_n$ given by $\chi_n(g)=n g n^{-1}$ is an isometry for all $n\in N$. The lemma above says that none of these automorphisms preserve any possible splitting $TN= \mv\,N \oplus
\mz\,N$ with the exception of the trivial ones, thus $\Iso^{spl}(N)\neq \Iso^{aut}(N)$ and $\Iso^{spl}(N)\neq \Iso (N)$ since $N$ is non abelian.

As already mentioned for bi-invariant metrics, the map: $g \to g^{-1}$ is an isometry which is not an automorphism, unless the group is abelian. Thus $\Iso^{aut}(N) \neq \Iso(N)$.
As stated above, the group $\Iso^{aut}(N)$ is not contained into the group $\Iso^{spl}(N)$. The next example shows that $\Iso^{spl}(N)$ could not be contained in $\Iso^{aut}(N)$.

\begin{exa} 
Consider  $\R^6$ with the canonical differentiable structure and let $g$ denote the following pseudo-Riemannian metric on $\RR^6$:
\begin{equation}\label{metric2} g=d x_1dx_6+dx_3dx_4-dx_2dx_5.
\end{equation}
This manifold admits an isometric transitive and simple action of the 2-step nilpotent Lie group $N$ which is modelled on $\R^6$ with multiplication operation such that for $p=(x_1,x_2,x_3,x_4,x_5,x_6)$ and $q=(y_1,y_2,$ $y_3,y_4,y_5,y_6)$ it holds
\begin{eqnarray}
p\cdot q&=&\left( x_1+y_1,x_2+y_2,x_3+y_3,x_4+y_4+\frac12 (x_1y_2-x_2y_1),\right.\nonumber\\
&&\;\quad\left. x_5+y_5+\frac12 (x_1y_3-x_3y_1), x_6+y_6+\frac12 (x_2y_3-x_3y_2)\right).
\end{eqnarray}
The corresponding metric on $N$ induced by (\ref{metric2})  is invariant under translations on the left and also on the right,  hence $g$  is a  bi-invariant (pseudo-Riemannian) metric on $N$. 
Its Lie algebra $\mn$ admits a basis $\{e_i\}_{i=1}^6$ obeying the  non-zero Lie-bracket relations $$[e_1,e_2]=e_4,\qquad [e_1,e_3]=e_5,\qquad [e_2,e_3]=e_6. $$
The Lie algebra $\mn$ is the free 2-step nilpotent Lie algebra on three generators and $N$ is its corresponding simply connected Lie group. At the Lie algebra level, the bi-invariant metric (\ref{metric2}) induces the  ad-invariant metric  verifying \cite{dBO}
  $$\lela e_1,e_6\rira=\lela e_3,e_4\rira =-\lela e_2,e_5\rira=1.$$
  
  The center $\mz$ of $\mn$ is spanned by $e_4,e_5,e_6$ and it is totally isotropic ($\mz^\bot=\mz$\,). Moreover
  \begin{equation}\mn=\mv\oplus \mz \label{spl} \end{equation}
   where $\mv=span \{e_1,e_2,e_3\}$ is also a totally isotropic subspace.
The splitting of $TN$ by left-invariant distributions associated to (\ref{spl}) at each point $p$ of $N$ is given by 
 $$\mv N_p=span\{X_1|_p,X_2|_p,X_3|_p\} \quad\mbox{ and }\quad\mz {N}_p=span\{X_4|_p,X_5|_p,X_6|_p\}$$
where $X_i$ is the left-invariant vector field of $N$ such that ${X_i}_{|_0}=e_i$. Canonical computations show that $\lela \partial_i,\partial_j\,\rira=\lela X_i,X_j\,\rira$ for all $i,j$, where $\partial_i=\frac{\partial}{\partial x_i}$. 

For each $\tau\in \RR$ consider the  diffeomorphism $F^\tau$ defined as
$$\begin{array}{l}
F^{\tau}(x_1,x_2,x_3,x_4,x_5,x_6)=(\cosh \tau\, x_1+\sinh \tau\,x_3,x_2,\sinh \tau\, x_1+\cosh \tau\, x_3,\\\hspace{4.2cm}\cosh \tau\, x_4-\sinh \tau\, x_6,x_5,
-\sinh \tau\, x_4+\cosh \tau\, x_6)\end{array}$$
which is an isometry of $N$. Indeed,  its differential at $p$, in the ordered basis $\{\partial_1,\,\partial_3,\,\partial_4,\,\partial_6,\,\partial_2,\,\partial_5\}$ of $T_pN$, is
$$dF^{\tau}_p= \left(
\begin{array}{ccc}
A&0&0\\
0&A^{-1}&0\\
0&0&I_{2\times 2}\end{array}\right)\quad\mbox{ where }\quad A=\left(
\begin{array}{cc}
\cosh \tau &\sinh \tau\\
\sinh \tau &\cosh \tau\end{array}\right)\in \Or(1,1).$$

The map $F^{\tau}$ preserves the metric in (\ref{metric2}) for all $\tau \in \RR$. Moreover, it preserves the splitting $TN=\mv N\oplus \mz N $. In fact $dF^{\tau}_p (\mz {N}_p)=\mz {N}_{F^\tau (p)}$ and
\begin{eqnarray}
dF^{\tau}_pX_1|_p
&=&\cosh \tau\,X_1|_{F^{\tau}(p)}+\sinh \tau\, X_3|_{F^{\tau}(p)}\in\mv N_{F^\tau(p)},\nonumber\\
dF^{\tau}_pX_2|_p&=&X_2|_{F^{\tau}(p)}\in \mv N_{F^{\tau}(p)},\nonumber\\
dF^{\tau}_pX_3|_p
&=&\sinh \tau\,X_1|_{F^{\tau}(p)}+\cosh \tau\, X_3|_{F^{\tau}(p)}\in\mv {N}_{F^{\tau}(p)}.\nonumber
\end{eqnarray}
Therefore $dF^{\tau}_p(\mv {N}_p) =\mv {N}_{F^{\tau}(p)}$ and $F^{\tau}\in \Iso^{spl}(N)$  

Finally, notice that $dF^{\tau}_0$ is not a Lie algebra isomorphism if $\tau\neq 0$, therefore $F^{\tau}\notin \Iso^{aut}(N)$  for $\tau\neq 0$. Furthermore, $\Iso(N)=N\cdot \Or(3,3)$ by Lemma \ref{iso}. 
\end{exa}

We have already proved that on the Lie group $N$ of the previous example  there
 are isometries preserving a fixed splitting which are not automorphisms.

\begin{pro}Let $(N, \bil)$ denote  a 2-step nilpotent Lie group endowed with a bi-invariant metric. Then the center is degenerate \cite{dBO} and
\begin{itemize}
\item $\Iso^{spl}(N)\neq\Iso^{aut}(N)$;
\item $\Iso^{spl}(N)\subsetneq\Iso(N)$ and $\Iso^{aut}(N)\subsetneq\Iso(N)$.
\end{itemize}
\end{pro}


\begin{thebibliography}{GGGG}
 

\bibitem{WZ} {\sc Z. Chen, J. A. Wolf}, {\it Pseudo-Riemannian weakly symmetric manifolds},
 Ann. Glob. Anal. Geom. {\bf 41}, 381--390 (2012).

\bibitem{Ci} {\sc P. Ciatti}, {\it Scalar products on Clifford modules and pseudo-$H$-type 
Lie algebras.}, Ann. Mat. Pura Appl., IV. Ser., {\bf 178},  1--31 (2000).

\bibitem{CP} { \sc L. Cordero and P. Parker}, {\it Isometry groups of pseudoriemannian 2-step 
nilpotent Lie groups}, Houston J. Math. {\bf 35} (1), 49--72 (2009).

\bibitem{CP2} { \sc L. Cordero and P. Parker}, {\it Lattices and periodic geodesics in 
pseudoriemannian 2-step nilpotent Lie groups}. Int. J. Geom. Methods Phys. {\bf 5} (1), 79--99 (2008).

\bibitem{dBO} {\sc V. del Barco and G.P. Ovando}, {\it Free nilpotent Lie algebras admitting 
ad-invariant metrics}, J. Algebra, {\bf 366}, 205--216  (2012).

\bibitem{BOV1} {\sc V. del Barco, G.P. Ovando and F. Vittone}, {\it Naturally reductive 
pseudo-Riemannian Lie groups in low dimensions}, arxiv:1211.0884.

\bibitem{DI} {\sc D. C. Duncan and E. C. Ihrig}, {\it Flat pseudo-Riemannian manifolds with 
a nilpotent transitive group of isometries}, Ann. Global Anal. Geom. {\bf 10} (1), 87--101 (1992).

\bibitem{Du} {\sc  Z. Dusek}, {\it Survey on homogeneous geodesics}, Note Mat. {\bf 1} 
(suppl. no. 1),  147--168  (2008).
                           
\bibitem{Eb} {\sc P. Eberlein}, {\it Geometry of {$2$}-step nilpotent groups with a left invariant 
metric}, Ann. Sci. \'Ecole Norm. Sup. (4), {\bf 27} no. 5 , 611--660 (1994).
     
\bibitem{Fi} {\sc J. Figueroa-O'Farill}, {\it Lorentzian symmetric spaces in supergravity}, 
Alekseevsky, Dmitri V. (ed.) et al., Recent developments in pseudo-Riemannian geometry. 
ESI Lectures in Mathematics and Physics, 419--494 (2008).

\bibitem{Go} {\sc C. Gordon}, {\it Transitive riemannian isometry groups
with nilpotent radicals}, Ann. Inst. Fourier {\bf 31} (2), 193--204 (1981).

\bibitem{JPP}{\sc C. Jang, P. Parker and K. Park}, {\it Pseudo $H$-type 2-step nilpotent Lie groups.},
Houston J. Math. {\bf 31} no. 3, 765--786  (2005).

\bibitem{Ka1} {\sc A. Kaplan}, {\it Riemannian nilmanifolds attached to {C}lifford modules}, 
Geom. Dedicata {\bf 11} no. 2 , 127--136 (1981).

\bibitem{Je} {\sc G. Jensen}, {\it The scalar curvature of left-invariant Riemannian metrics},  Indiana U. Math. J. {\bf 20}, 1125--1144 (1971).

\bibitem{Mu} {\sc D. M\"uller}, {\it Isometries of bi-invariant pseudo-Riemannian metrics 
on Lie groups},  Geom. Dedicata  {\bf 29} no. 1, 65--96  (1989).

\bibitem{NW} {\sc C. R. Nappi and E. Witten}, {\it Wess-Zumino-Witten model based on a  nonsemisimple group}, Phys. Rev. Lett. {\bf 71} (23), 3751--3753 (1993).
                              
\bibitem{ON} {\sc B. O'Neill}, {\it Semi-Riemannian geometry with applications to relativity},
 Academic Press (1983).

\bibitem{Ov} {\sc G. Ovando}, {\it Naturally reductive pseudo Riemannian 2-step nilpotent
 Lie groups},  Houston J. Math. {\bf 39} (1), 147--168 (2013). 


\bibitem{Ov3} {\sc G. Ovando}, {\it Two-step nilpotent Lie algebras with ad-invariant metrics 
and a special kind of skew-symmetric maps}, J. Algebra and its Appl., {\bf 6} (6), 897--917  (2007).

\bibitem{St} {\sc R. F. Streater}, {\it The representations of the oscillator group}, Comm. Math. 
Phys. {\bf 4} (3), 217--236 (1967).

\bibitem{Wi} {\sc E. Wilson}, {\it Isometry groups on homogeneous nilmanifolds}, Geom. Dedicata 
{\bf 12} no. 2, 337--345 (1982).

\bibitem{Wo} {\sc J. Wolf}, {\it On Locally Symmetric Spaces of Non-negative
Curvature and certain other Locally Homogeneous Spaces}, Comment. Math. Helv. {\bf 37}, 266--295 (1962-63).

\end{thebibliography}
\end{document}